\documentclass[10pt,reqno]{amsart}
\usepackage{amssymb,mathrsfs,graphicx}
\usepackage{ifthen}
\usepackage{colortbl}
\usepackage{cases,appendix}
\definecolor{black}{rgb}{0.0, 0.0, 0.0}
\definecolor{red}{rgb}{1.0, 0.5, 0.5}

\provideboolean{shownotes} 
\setboolean{shownotes}{true} 
\newcommand{\margnote}[1]{
\ifthenelse{\boolean{shownotes}}%
{\marginpar{\raggedright\tiny\texttt{#1}}}%
{}%
}
\newcommand{\hole}[1]{
\ifthenelse{\boolean{shownotes}}%
{\begin{center} \fbox{ \rule {.25cm}{0cm} \rule[-.1cm]{0cm}{.4cm}
\parbox{.85\textwidth}{\begin{center} \texttt{#1}\end{center}} \rule
{.25cm}{0cm}}\end{center}} {} }

\title[Propagation of chaos for the VPFP equation with a polynomial cut-off]{Propagation of chaos for the Vlasov-Poisson-Fokker-Planck equation with a polynomial cut-off}

\author[Carrillo]{Jos\'{e} A. Carrillo}
\address[Jos\'{e} A. Carrillo]{\newline Department of Mathematics
    \newline Imperial College London, London SW7 2AZ, United Kingdom}
\email{carrillo@imperial.ac.uk}

\author[Choi]{Young-Pil Choi}
\address[Young-Pil Choi]{\newline Department of Mathematics and Institute of Applied Mathematics\newline
Inha University, Incheon 402-751, Republic of Korea}
\email{ypchoi@inha.ac.kr}

\author[Salem]{Samir Salem}
\address[Samir Salem]{\newline CEREMADE \newline
     Paris Dauphine, Place du Mar\'echal De Lattre De Tassigny · 75775 Paris CEDEX 16, France}
\email{salem@ceremade.dauphine.fr}

\numberwithin{equation}{section}

\newtheorem{theorem}{Theorem}[section]
\newtheorem{lemma}{Lemma}[section]

\newtheorem{proposition}{Proposition}[section]
\newtheorem{remark}{Remark}[section]

\newcommand{\R}{\mathbb R}
\newcommand{\N}{\mathbb N}

\newcommand{\VV}{\mathcal{V}}
\newcommand{\WW}{\mathcal{W}}
\newcommand{\XX}{\mathcal{X}}
\newcommand{\YY}{\mathcal{Y}}

\newcommand{\pp}{\mathcal P}
\newcommand{\e}{\varepsilon}

\newcommand{\lal}{\langle}
\newcommand{\ral}{\rangle}
\newcommand{\mc}{\mathcal{C}}

\newcommand{\lt}{\left}
\newcommand{\rt}{\right}
\newcommand{\pa}{\partial}
 
\newcommand{\mb}{\mathbf{1}}

\newcommand{\bq}{\begin{equation}}
\newcommand{\eq}{\end{equation}}

\newcommand{\LL}{\mathcal{L}}

\newcommand{\E}{\mathbb{E}}
\def\charf {\mbox{{\text 1}\kern-.30em {\text l}}}

\begin{document}
\allowdisplaybreaks

\date{\today}



\begin{abstract} We consider a $N$-particle system interacting through the Newtonian potential with a polynomial cut-off in the presence of noise in velocity. We rigorously prove the propagation of chaos for this interacting stochastic particle system. Taking the cut-off like $N^{-\delta}$ with $\delta < 1/d$ in the force, we provide a quantitative error estimate between the empirical measure associated to that $N$-particle system and the solutions of the $d$-dimensional Vlasov-Poisson-Fokker-Planck system. We also study the propagation of chaos for the Vlasov-Fokker-Planck equation with less singular interaction forces than the Newtonian one.
\end{abstract}

\maketitle \centerline{\date}


%
%
%
%
\section{Introduction}\label{intro}

The starting point are the classical Newton dynamics for point particles interacting through the interaction force $F$ in the presence of noise:
\bq\label{par_VP}
\left\{ \begin{array}{ll}
dX_t^i=V_t^idt,  & \\[1mm]
\displaystyle dV_t^i=\frac{1}{N}\sum_{j=1}^N F(X_t^i-X_t^j)\,dt+\sqrt{2\sigma} dB^i_t, & 
\end{array} \right. \quad i=1,\cdots,N, \quad t > 0,
\eq
where $(\mathcal{X}_t,\mathcal{V}_t):=(X_t^1,\cdots,X_t^N,V_t^1,\cdots V_t^N)\in \R^{2dN}$ are positions and velocities of the particles in $\R^d$, and $\sigma > 0$ is the noise strength. Here, $\{ \lt(B^i_t\rt)_{t\geq 0}\}_{i=1}^N$ are $N$ independent $d$-dimensional Brownian motions constructed over some probability space not specified for simplicity. There are many examples of physical systems of the form \eqref{par_VP}. Most classical examples of interaction kernels are the Coulombian or gravitational forces:
\bq\label{cou}
F(x) = \xi\frac{x}{|x|^{d}} \quad \mbox{where} \quad \xi = \left\{ \begin{array}{ll} +1 & \mbox{for plasma problems}\\
-1 & \mbox{for astrophysics problems}.
\end{array} \right.
\eq
We are interested in the rigorous derivation of the Vlasov-Poisson-Fokker-Planck (VPFP) equation as $N \to \infty$:
\begin{equation}
\label{eq:VPFP}
\left\{ \begin{array}{ll}
\displaystyle \pa_t f_t+v\cdot\nabla_x f_t+ (F*\rho_t)\cdot\nabla_v f_t=\sigma\Delta_v f_t, \quad (x,v) \in \R^d \times \R^d, \quad t > 0,& \\[2mm]
\displaystyle \rho_t(x)=\int_{\R^d} f_t(x,v)\,dv,  
\end{array} \right.
\end{equation}
with the Poisson force $F(x) = \xi x/|x|^{d},\xi = \pm 1$, where $f = f(x,v,t)$ is the probability density function at phase space $(x,v)$ at time $t$. Notice that the VPFP system is the formal mean-field limit equation associated to \eqref{par_VP}. This is a classical kinetic equation whose well-posedness and qualitative properties were studied in different settings: classical solutions \cite{VO,Bou}, weak solutions \cite{BD,CS1,CSV,Vi,He} and the references therein. We refer to \cite{G,Pe} for general theory on these problems.

The rigorous proof of the mean-field limit in the one dimensional case for \eqref{eq:VPFP} with the force $F$ given in \eqref{cou} was obtained in \cite{Hau, HS}. In \cite{JW}, the propagation of chaos for the system \eqref{eq:VPFP} with bounded forces, i.e., $F \in L^\infty(\R^d)$ is studied based on relative entropy arguments. The intimately related classical Vlasov-Poisson system has also received lots of attention in the last years. This system corresponds to \eqref{par_VP} and \eqref{eq:VPFP} with $\sigma=0$.
The propagation of chaos has been shown in \cite{HJ1, HJ2} for the less singular case, i.e., $|F(x)| \leq |x|^{-\alpha}$ with $\alpha < 1$. More recently the physically relevant case of the propagation of chaos for the Newtonian potential with polynomial cut-off has been obtained in \cite{LP}. 

The propagation of chaos and its consequence, the rigorous proof of the mean-field limit, are questions of nowadays importance in many other problems; for instance, in kinetic models of collective behavior \cite{BCC,CCR,CCHS,CS,CS2,CSp} with or without noise in velocity, see also \cite{Dob,Mc} for more general types of equations. However, in most of these applications the singularity of the kernels is much better behaved. The case of Holder interaction with Holder exponent greater than $2/3$ was also recently treated in \cite{Hol}. The main motivation of our present work is to obtain the propagation of chaos in the degenerate diffusion setting, as the noise acts only on velocity, in the classical case of Newtonian interaction with cut-off. 

Finally, let us point out that most of the propagation of chaos results for nonlinear nonlocal conservation equations with non Lipschitz interaction, concern first order models. Without diffusion, the mean field limit for the aggregation equation with interaction less singular than the Newtonian is studied in \cite{CCH}. With diffusion, the case of Holder interaction is treated in \cite{Hol} by quantitative means leading to explict rates of convergence. For the 2D Navier-Stokes equation in  vortex formulation, the propagation of chaos is obtained in \cite{MSN} by compactness, thanks to result of \cite{HM}. Similar techniques have been applied to the case of the diffusion dominated 2D Keller-Segel equations in \cite{GQ,S} and geometrical constraints interactions with reflecting diffusions in \cite{CS}.

Let us be more precise about the different approximation levels involved in the propagation of chaos. Consider the regularized particle system
\begin{equation}\label{eq:Nps_CO}
\left\{ \begin{array}{ll}
\displaystyle dX_t^{i,N}=V_t^{i,N}dt,     & \\[2mm]
\displaystyle dV_t^{i,N}=\frac{1}{N}\sum_{j=1}^N F^N_\delta (X_t^{i,N}-X_t^{j,N})\,dt+\sqrt{2\sigma} dB^{i,N}_t, & 
\end{array} \right. \quad i=1,\cdots,N, \quad t > 0,
\end{equation}
with the initial data $(X^{i,N}(0), V^{i,N}(0)) = (X^{i,N}_0, V^{i,N}_0)$ for $i=1,\cdots,N$, where the interaction force $F^N_\delta, \delta > 0$ is given by
\[
F^N_\delta(x) := \xi\frac{x}{(|x| \vee N^{-\delta})^d} \quad \mbox{and} \quad F^N_\delta(0) = 0,
\]
where $ a \vee b := \max\{a,b \}$. Note that this system of stochastic differential equations (SDE) \eqref{eq:Nps_CO} has a unique strong solution by the standard theorem for SDEs. It is worth mentioning that the sign of force term is not important for the propagation of chaos provided that we analyze it on a finite time interval. Thus we only focus on the case $\xi = 1$.

Our main purpose is to provide a qualitative error estimates between the solutions to the particle system \eqref{eq:Nps_CO} and the VPFP equation \eqref{eq:VPFP}. For this, we need to introduce an intermediate system of independent copies of {\it nonlinear SDEs with cut-off} given by
\begin{equation}\label{eq:NLps_CO}
\left\{ \begin{array}{ll}
dY_t^{i,N}=W_t^{i,N}dt, \quad i=1,\cdots, N, \quad t > 0, & \\[2mm]
\displaystyle dW_t^{i,N}=F^N_{\delta}*\rho_t^N(Y_t^{i,N})dt+\sqrt{2\sigma }dB^{i,N}_t,  \quad \rho^N_t=\LL(Y_t^{i,N}), &  \\[2mm]
(Y^{i,N}_0, W^{i,N}_0) = (X^{i,N}_0, V^{i,N}_0) \quad \mbox{for} \quad i =1,\cdots,N, &  
\end{array} \right.
\end{equation}
where $\rho_t^N = \int_{\R^d} f_t^N\,dv$, and $f_t^N=\LL(Y_t^{i,N},W_t^{i,N})$, for all $i=1,\cdots, N$, is the global-in-time weak solution to the regularized VPFP system
\bq\label{reg_VPFP}
\pa_t f_t^N+v\cdot\nabla_x f_t^N+ (F^N_\delta *\rho_t^N)\cdot\nabla_v f_t^N=\sigma\Delta_v f_t^N, \quad (x,v) \in \R^d \times \R^d, \quad t > 0,
\eq
with the initial data $f_0^N = \LL(Y_0^{i,N},W_0^{i,N})$, $i=1,\cdots,N$. Due to the regularization, we can easily obtain the global existence of weak solutions to the equation \eqref{reg_VPFP} under suitable assumptions on the initial data. Solving the system \eqref{eq:NLps_CO} with the given $\rho_t^N$ to get the existence and uniqueness of solutions to the system \eqref{eq:NLps_CO} is by now standard.

Before stating our main result in this paper, we need to introduce some notions and notations. We define the empirical measure $\mu_t^N$ associated to a solution to the particle system \eqref{eq:Nps_CO} and $\nu_t^N$ the one associated to the nonlinear independent particle \eqref{eq:NLps_CO} as
\[
\mu^N_t = \frac1N \sum_{i=1}^N\delta_{(X_t^{i,N},V_t^{i,N})} \quad \mbox{and} \quad \nu^N_t = \frac1N \sum_{i=1}^N\delta_{(Y_t^{i,N},W_t^{i,N})},
\]
respectively. For a function $f$ and $p \in [1,\infty]$, $\|f\|_p$ represents the usual $L^p$-norm for functions in $L^p(\R^{2d})$, and $\|f\|_{p,q}:=\|f\|_p + \|f\|_q$ for $p,q\in[1,\infty]$. For $T > 0$,  $L^p(0,T,E)$ is the set of $L^p$ functions from an interval $(0,T)$ to a Banach space $E$. $\pp(\R^{2d})$ and $\pp_p(\R^{2d})$ stand for the sets of all probability measures and probability measures with finite moments of order $p \in [1,\infty)$. We consider the optimal transport distances $\WW_p$, $p \in [1,\infty]$, on $\pp_p(\R^{2d})$, see \cite{Villani} for classical definitions and properties of optimal transport distances. We also denote by $C$ a generic positive constant independent of $N$. 

The main result of this work gives an explicit estimate on the decay of the distance between the empirical measure of the interacting particle system $\mu^N_t$ and the solution of the VPFP equation $f_t$ in optimal transport distances in the probability sense.

\begin{theorem}\label{thm:PropChao}
Let $T > 0, d > 1$ and $\delta < 1/d$. Let $(X_0^{i,N}, V_0^{i,N})_{i=1,\cdots,N}$ be $N$ independent random variables with law $f_0$. Let $f_t$ and $f_t^N$ be the solutions to the VPFP equation \eqref{eq:VPFP} and its regularized version \eqref{reg_VPFP}, respectively, up to time $T > 0$. Assume further that $f, f^N \in L^\infty(0,T; (L^1 \cap L^\infty)(\R^{2d})) \cap \mc([0,T]; \pp_q(\R^{2d}))$ with the same initial data $f_0 \in (L^1 \cap L^\infty \cap \pp_q)(\R^{2d})$ for some $q \geq 2$, and that their respective spatial densities $\rho_t$ and $\rho^N_t$ satisfy
\[
\int_0^T \|\rho_t\|_{\infty}\,dt \leq C_0 \quad \mbox{and} \quad \sup_{N\in \mathbb{N}}\int_0^T \|\rho^N_t\|_{\infty}\,dt \leq C_0,
\]
with $C_0$ independent of $N$. 
Then, given any $p \in [1,2q)$, $\gamma < \lt( \frac{1}{2(d \vee  p)} \wedge \delta \rt)$, and $\e \in \lt(0, q - \frac{p}{1 - p \gamma} \rt)$, the estimate 
\[
\sup_{0 \leq t \leq T}\mathbb{P}\lt(\WW_p(\mu_t^N,f_t)\geq CN^{-\gamma} \rt) \leq  C N^{\frac1p\lt(1 - \frac{(1-p\gamma)(q - \e)}{p} \rt)} + C_N, \qquad \mbox{for $N$ large enough},
\]
holds for some constant $C>0$ depending only on $d,T,p,q,\e, f_0$, and $C_0$. Here $C_N$ is given by
\[
-\log C_N=\begin{cases}
\displaystyle \frac Cp N^{1 - 2p\gamma} & \text{ if } p>d, \\[1mm] 
\displaystyle \frac{CN^{1 - 2p\gamma}}{p (\ln(2 + N^{p\gamma}))^2} & \text{ if } p=d, \\[4mm] 
\displaystyle \frac Cp N^{1 - 2d\gamma} & \text{ if } 1 \leq p<d.
\end{cases}
\]

\end{theorem}

Let us briefly explain the strategy of the proof of this theorem.
\begin{itemize}
\item First, using the techniques introduced in \cite{LP}, one can estimate the probability that the error between the empirical measure $\mu_t^N$ associated to the particle system \eqref{par_VP} and the empirical measure $\nu_t^N$ associated to the nonlinear independent particle system with cut-off \eqref{eq:Nps_CO}, exceeds the threshold $N^{-\gamma}$. In this way, it is shown that this probability decreases faster than any negative power of $N$ (see Lemma \ref{lem_j}). 
\item Then using a concentration inequality of \cite{FG}, one can obtain bounds on the probability that the error between $\nu_t^N$ empirical measure associated to i.i.d. random variables and their law $f_t^N$ solution at time $t$ to equation \eqref{reg_VPFP}, exceeds the threshold $N^{-\gamma}$. The main results in \cite{FG} provide an optimal rate of convergence which is of order of some negative power of $N$. (See Proposition \ref{prop_fg}). 

\item Finally, we show that the error between the solution to equation \eqref{reg_VPFP} and the solution to equation \eqref{eq:VPFP} never exceeds the threshold $N^{-\gamma}$, for $N$ large enough. Contrary to \cite[Proposition 9.1]{LP}, this error has to be estimated in $\WW_p$ distance with $1\leq p<\infty$ due to the presence of noise in velocity (See Proposition \ref{lem:log_lip}).
\end{itemize}

Our main contributions in comparison to the noiseless case treated in \cite{LP} are the following. On the one hand, we provide a convergence result of the solution to equation \eqref{reg_VPFP} to the solution to equation \eqref{eq:VPFP}, in $\WW_p$ metric for $p\in [1,\infty)$, which is crucial since the support of $f^N_t$ and $f_t$ are not compactly supported in our present case. On the other hand, we provide a well-posedness result for equation \eqref{eq:VPFP}. This requires to show that the spatial density $\rho_t$ of solution to this equation lies in $L^1(0,T;L^\infty(\R^d))$. For the case without diffusion, i.e., Vlasov-Poisson system, characteristic methods can be used to get a uniform bound on the spatial density, which leads to a uniqueness of solutions, under suitable assumptions on the initial density \cite{LP1,Loe,Pal}. However, the presence of diffusion makes it more complicated. In \cite{PS}, an uniform-in-time $L^\infty$-bound of the spatial density is obtained by means of the stochastic characteristic method under the assumptions on compactly supported initial density $f_0$ in velocity. We also provide a simple proof of the local-in-time $L^\infty$ propagation by employing Feynman-Kac's formula assuming only that the initial data has a polynomial decay in velocity (see Lemma \ref{lem:Linfbound} together with Theorem \ref{thm:ex-uniq}). Notice that obtaining the bound estimate of $\rho_t$ in $L^\infty(\R^d)$ is equivalent to the one for $\|\rho_t^N\|_\infty$ in $N$. 

The rest of this paper is organized as follows. In Section 2, we deal with some preliminary materials introduced in \cite{LP}. Section 3 contains the key new estimate on the $\WW_p$ stability between of the solutions of the VPFP equation with and without cut-off. Section 4 shows the well posedness of solution to the VPFP equation with the assumed regularity in Theorem \ref{thm:PropChao}, i.e., the assumed $L^\infty$ bound on the spatial density. Finally, Section 5 contains generalizations of this result for less singular kernels.

%
%

\section{Preliminaries}
In this section, we provide estimates for the force fields in the equations \eqref{eq:VPFP} and \eqref{reg_VPFP} in $L^1(0,T;L^\infty(\R^d))$ and the $q$-th moment estimate for the solutions of that. We also recall several useful estimates on the force fields whose proofs can be found in \cite{LP}.

Under the assumptions on the spatial densities $\rho_t$ and $\rho_t^N$ in Theorem \ref{thm:PropChao}, we can easily find 
\[
\int_0^T \|F * \rho_t\|_\infty \,dt < \infty\quad \mbox{and} \quad \sup_{N\in \mathbb{N}}\int_0^T \|F^N_\delta * \rho^N_t\|_{\infty}\,dt <\infty,
\]
due to $\|F * \rho_t \|_\infty \leq C\|\rho_t\|_{1,\infty}$ and $\|F^N_\delta * \rho_t^N \|_\infty \leq C\|\rho_t^N\|_{1,\infty}$ with $C>0$ independent of $N$. These estimates together with straightforward computations yield that for some $q \geq 2$
$$\begin{aligned}
\frac{d}{dt}\int_{\R^{2d}} &(|x|^q + |v|^q) f\,dxdv \cr
&= q\int_{\R^{2d}} |x|^{q-2}x \cdot v f\,dxdv + q\int_{\R^{2d}} |v|^{q-2}v \cdot (F * \rho) f\,dxdv + \sigma q(q-2+d)\int_{\R^{2d}} |v|^{q-2} f\,dxdv\cr
&\leq C + C\int_{\R^{2d}} |x|^q f\,dxdv + C(1 + \|\rho_t\|_{1,\infty})\int_{\R^{2d}}|v|^q f\,dxdv.
\end{aligned}$$
Thus, we obtain
\[
\sup_{t \in [0,T]} \int_{\R^{2d}}(|x|^q+|v|^q)f_t \,dxdv \leq C\sup_{0 \leq t \leq T} \int_{\R^{2d}}(|x|^q+|v|^q)f_0 \,dxdv,
\]
and similarly, we also have
\[
\sup_{N\in \mathbb{N}} \sup_{t\in[0,T]} \int_{\R^{2d}}(|x|^q+|v|^q)f^N_t\,dxdv \leq C\sup_{0 \leq t \leq T} \int_{\R^{2d}}(|x|^q+|v|^q)f_0 \,dxdv,
\]
where $C > 0$ is independent of $N$. 

Let us now recall the interaction force with a cut-off:
\[
F^N_\delta(x) := \frac{x}{(|x| \vee N^{-\delta})^d},
\]
and define
$l_\delta^N(x)$ by
\[
l_\delta^N(x) := 
\left\{ \begin{array}{ll}
\displaystyle \frac{1}{|x|^d} & \mbox{if } |x| \geq dN^{-\delta},\\[4mm]
N^{d\delta} &  \mbox{otherwise}.
\end{array} \right.
\]
We drop the subscript $\delta$ in $F^N_\delta$ and $l^N_\delta$ for notational simplicity in the rest of paper, i.e., $F^N_\delta = F^N$ and $l^N_\delta = l^N$.
\begin{lemma}\label{lem_useful} Let $d > 1$ be given.

1. There exists a constant $C$, which depends only on $d$, such that 
\[
|F^N(x) - F^N(x+z)| \leq C l^N(x) |z|,
\]
for any $x,z \in \R^d$ with $|z|\leq (d-1)N^{-\delta}$.

2. There exists a constant $C>0$ independent of $N$ such that
\[
\|l^N * \rho\|_\infty \leq C\ln N \|\rho\|_{1,\infty} \quad \mbox{and} \quad \|\nabla F^N * \rho\|_\infty \leq C\ln N \|\rho\|_{1,\infty}.
\]
\end{lemma}
\begin{proof}
See Lemmas 6.1 and 6.3 of \cite{LP}.
\end{proof}

Then we recast here some law of large number like estimates. For $\kappa, \delta > 0$, we set $h:\R^d \to \R^d$ such that
\bq\label{eq_h}
|h(x)|\leq c_0 (N^{\kappa\delta} \wedge |x|^{-\kappa}).
\eq

\begin{lemma}\label{lem:LLN}
Let $(Y_1,\cdots,Y_N)$ be i.i.d. random variables of law $\rho\in L^{\infty}(\R^d)$, and define the associated empirical measure $\rho_N=\frac{1}{N}\sum_{i=1}^N\delta_{Y_i}$. Suppose that $0 < \e:=2\kappa\delta+(1-d\delta)\mb_{1> d\delta} < 2$. Then, for all integer $m>\frac{1}{2-\e}$, there exist $\gamma_m,C_m>0$ such that 
\[
\mathbb{E}\lt[\sup_{1 \leq i \leq N}  \left | h*\rho_N(Y_i)-h*\rho(Y_i) \right |^{2m}  \rt]\leq C_m N^{-\gamma_m},
\]
where $\gamma_m=(2-\e)m-1$.
\end{lemma}
\begin{proof}
See the proof of Proposition 7.2 of \cite{LP}.
\end{proof}
We conclude this section by recalling some concentration inequalities in the proposition below whose proof can be found in \cite[Theorem 2]{FG}.
\begin{proposition}\label{prop_fg}
	Let $N \geq 1$ and $(Y_1,\cdots,Y_N)$ be $N$ independent identically distributed of law $\rho \in \pp_q(\R^{2d})$ for some $q > 0$. Let $\rho_N := \frac1N \sum_{i=1}^N\delta_{Y_i}$. Then, for any $p\in (0,q/2),\e\in (0,q)$, there exist constants $C,c>0$ depending only on $d,p,q,\e$ and the $q$-order moment bound of $\rho$ such that 
	\[
	\mathbb{P}\bigl(\WW_p^p(\rho_N,\rho)\geq x \bigr)\leq CN (Nx)^{-\frac{q-\e}{p}}+a(N,x)\mb_{x\leq 1}, \quad \mbox{for any} \quad x > 0,
	\]
	where $a(N,x)$ is defined by
	\[
	-\log a(N,x)=\begin{cases}
	cNx^2 & \text{ if } p>d, \\ 
	cN(x/\ln(2+1/x))^2 & \text{ if } p=d, \\ 
	cNx^{2d/p} & \text{ if } 1 \leq p<d.
	\end{cases}
	\]
\end{proposition}
%
%
%
%
\section{Propagation of Chaos: Proof of Theorem \ref{thm:PropChao}}
In this section, we provide the details of the proof of Theorem \ref{thm:PropChao}. We begin with the following Lipschitz estimates on the force fields $F$ and $F^N$.
\begin{lemma}\label{lem_lip} For $x,y \in \R^d$, we get
	\[
	|F(x) - F(y)| \leq C_0|x-y|\lt( \frac{1}{|x|^d} + \frac{1}{|y|^d}\rt),
	\]
	and
	\[
	|F^N(x) - F^N(y)| \leq C_0|x-y|\lt( \frac{1}{\lt( |x| \vee N^{-\delta} \rt)^d} + \frac{1}{\lt( |y| \vee N^{-\delta} \rt)^d}\rt),
	\]
	where $C_0$ is a positive constant depending only on $d$.
\end{lemma}
\begin{proof} The first assertion is straightforward. For the proof of the second one, we consider three cases as follows. \newline
	
	(i) $|x|, |y| \geq N^{-\delta}$: In this case, we get $F^N(x) = F(x)$ and thus it is clear to obtain
	\[
	|F^N(x) - F^N(y)| = |F(x) - F(y)| \leq \lt( |\nabla F(x)| \vee |\nabla F(y)| \rt)|x-y|.
	\]
	Note that
	\[
	\nabla F(x) = \frac{1}{|x|^d} I_d - d\frac{x \otimes x}{|x|^{d+2}}, \quad \mbox{i.e.,} \quad |\nabla F(x)| \leq \frac{C}{|x|^d} = \frac{C}{\lt( |x| \vee N^{-\delta} \rt)^d},
	\]
	where $C > 0$ depends only on $d$. This yields
	\[
	|F^N(x) - F^N(y)| \leq \frac{C|x-y|}{\lt( |x| \wedge |y| \rt)^d} \leq C|x-y|\lt( \frac{1}{\lt( |x| \vee N^{-\delta} \rt)^d} + \frac{1}{\lt( |y| \vee N^{-\delta} \rt)^d}\rt),
	\]
	where $a \wedge b := \min \{a,b\}$. 
	
	(ii) $|x|, |y| \leq N^{-\delta}$: By definition of $F^N$, we find 
	\[
	|F^N(x) - F^N(y)| = \frac{|x-y|}{N^{-\delta d}} \leq |x-y|\lt( \frac{1}{\lt( |x| \vee N^{-\delta} \rt)^d} + \frac{1}{\lt( |y| \vee N^{-\delta} \rt)^d}\rt).
	\]
	
	(iii) $|x| < N^{-\delta} \leq |y|$ or $|y| < N^{-\delta} \leq |x|$: For $|x| < N^{-\delta} \leq |y|$, let us define $\tilde x$ as the interaction between the line segment $[x,y]$ and the ball $B(0,N^{-\delta})$. Then we get 
	\[
	|x| < |\tilde x| = N^{-\delta} \leq |y| \quad \mbox{and} \quad |x - \tilde x| + |\tilde x - y| = |x - y|.
	\]
	By employing the similar arguments as in the previous cases (i) and (ii), we have
	\[
	|F^N(x) - F^N(\tilde x)| \leq \frac{C|x-\tilde x|}{N^{-\delta d}}\quad \mbox{and} \quad |F^N(\tilde x) - F^N(y)| \leq \frac{C|\tilde x-y|}{\lt( |\tilde x| \wedge |y| \rt)^d} = \frac{C|\tilde x - y|}{|\tilde x|^d} = \frac{C|\tilde x - y|}{N^{-\delta d}}.
	\]
	This implies that for $|x| < N^{-\delta} \leq |y|$
	\[
	|F^N(x) - F^N(y)| \leq |F^N(x) - F^N(\tilde x)| + |F^N(\tilde x) - F^N(y)| \leq \frac{C|x- y|}{N^{-\delta d}} = \frac{C|x-y|}{\lt(|x| \vee N^{-\delta} \rt)^d}.
	\]
	Similarly, we have
	\[
	|F^N(x) - F^N(y)| \leq \frac{C|x-y|}{\lt(|y| \vee N^{-\delta} \rt)^d} \quad \mbox{for} \quad |y| < N^{-\delta} \leq |x|.
	\]
	Combining the above all cases, we conclude to the desired result.
\end{proof} 

Set
\[
G:=|X - Y| + |V - W| \quad \mbox{and} \quad G_N:= \sqrt{\ln N}|X - Y| + |V - W|.
\]

\begin{proposition}\label{lem:log_lip}
Let $(X,V)$ and $(Y,W)$ be two random variables of law $f_1$ and $f_2$ respectively, such that their first marginal are $\rho_1$ and $\rho_2$, respectively. Let $(\overline{X},\overline{Y})$ be an independent copy of $(X,Y)$. Suppose $\rho_1,\rho_2\in L^{\infty}(\R^d)$ and $N \geq e$. Then we have
\[
\E\lt[|F(X-\overline{X})-F(Y-\overline{Y})|G^{p-1}  \rt]\leq C\lt(\|\rho_1\|_{\infty}+\|\rho_2\|_{\infty}\rt)\E[G^p]\lt(1- \frac{1}{p}\ln^- \E[G^p] \rt)
\]
and
\[
\E\lt[|F^N(X-\overline{X})-F^N(Y-\overline{Y})| G_N^{p-1}  \rt]\leq C\sqrt{\ln N}\lt(\|\rho_1\|_{\infty}+\|\rho_2\|_{\infty}\rt)\E[G_N^p],
\]
for $p \geq 1$, where $\ln^{-}$ denotes the negative part of $\ln$ and the constant $C > 0$ depends only on $d$. 
\end{proposition}
\begin{proof} 
Let us denote by $\alpha_d$ the surface area of unit ball in $\R^d$.

{\it Estimate for the non cut-off force field.-} First, we notice that
\[
|F(X-\overline{X})-F(Y-\overline{Y})| \leq 
\left\{ \begin{array}{ll}
\displaystyle C \lt(|X-Y|+|\overline{X}-\overline{Y}|\rt)\lt(\frac{1}{|X-\overline{X}|^d}+\frac{1}{|Y-\overline{Y}|^d}\rt), & \\[3mm]
\displaystyle \frac{1}{|X-\overline{X}|^{d-1}}+\frac{1}{|Y-\overline{Y}|^{d-1}}. &
\end{array} \right.
\]
This yields that for any $r > 0$
$$\begin{aligned}
\E & \lt[|F(X-\overline{X})-F(Y-\overline{Y})| G^{p-1} \rt]\cr
&\quad \leq \E\lt[\lt(\frac{1}{|X-\overline{X}|^{d-1}}+\frac{1}{|Y-\overline{Y}|^{d-1}}\rt)\mb_{|X-\overline{X}|\wedge |Y-\overline{Y}|\leq r}G^{p-1} \rt]\\
&\quad +C\E\lt[\lt(|X-Y|+|\overline{X}-\overline{Y}|\rt)\lt(\frac{1}{|X-\overline{X}|^d}+\frac{1}{|Y-\overline{Y}|^d}\rt)\mb_{|X-\overline{X}|\wedge|Y-\overline{Y}|> r}G^{p-1} \rt]\\
&=: I_1+I_2.
\end{aligned}$$
$\diamond$ Estimate of $I_1$: Note that the event $\{|X-\overline{X}|\wedge |Y-\overline{Y}| \leq  r \}$ can be partitioned as
$$\begin{aligned}
&\{|X-\overline{X}|\wedge |Y-\overline{Y}| \leq r \}\cr
&\quad =\{|X-\overline{X}|\vee |Y-\overline{Y}| \leq r \} \cup \{|X-\overline{X}|> r \geq |Y-\overline{Y}| \} \cup \{|X-\overline{X}| \leq r < |Y-\overline{Y}| \}. 
\end{aligned}$$
Taking account this, we split $I_1$ into three terms:	
$$\begin{aligned}
I_1&=\E\lt[\lt(\frac{1}{|X-\overline{X}|^{d-1}}+\frac{1}{|Y-\overline{Y}|^{d-1}}\rt)\mb_{|X-\overline{X}|\vee |Y-\overline{Y}| \leq r} G^{p-1} \rt]\\
&\quad + \E\lt[\lt(\frac{1}{|X-\overline{X}|^{d-1}}+\frac{1}{|Y-\overline{Y}|^{d-1}}\rt)\mb_{|X-\overline{X}|> r \geq |Y-\overline{Y}|}  G^{p-1} \rt]\\
&\quad + \E\lt[\lt(\frac{1}{|X-\overline{X}|^{d-1}}+\frac{1}{|Y-\overline{Y}|^{d-1}}\rt)\mb_{|X-\overline{X}| \leq r < |Y-\overline{Y}|}  G^{p-1} \rt]\\
&=: I^1_1+I^2_1+I^3_1.		
\end{aligned}$$
For the estimate of $I_1^1$, we get
$$\begin{aligned}
I_1^1&=\E_{(X,V),(Y,W)}\lt[ \E_{\overline{X},\overline{Y}}\lt[ \lt(\frac{1}{|X-\overline{X}|^{d-1}}+\frac{1}{|Y-\overline{Y}|^{d-1}}\rt)\mb_{|X-\overline{X}|\vee |Y-\overline{Y}| \leq r} G^{p-1}  \rt]   \rt]\\
&\leq \E_{(X,V),(Y,W)}\lt[\lt(\int_{|X-x|\leq r} \frac{1}{|X-x|^{d-1}}\rho_1(dx)+\int_{|Y-y|\leq r} \frac{1}{|Y-y|^{d-1}}\rho_2(dy)\rt) G^{p-1}\rt]\\
&\leq \alpha_d\lt( \|\rho_1\|_{\infty}+ \|\rho_2\|_{\infty}\rt)r\E \bigl[ G^{p-1}\bigr].
\end{aligned}$$
For $I_1^2$, we obtain
\[
\begin{aligned}
I_1^2&\leq \E\lt[\frac{2}{|Y-\overline{Y}|^{d-1}}\mb_{|Y-\overline{Y}| \leq r}G^{p-1}\rt]\\
&=2\E_{(X,V),(Y,W)}\lt[\lt(\int_{|Y-y|\leq r} \frac{1}{|Y-y|^{d-1}}\rho_2(dy)\rt) G^{p-1} \rt]\\
&\leq 2\alpha_d\|\rho_2\|_{\infty}r\E\lt[G^{p-1} \rt].
\end{aligned}
\]
Similarly, we estimate $I_1^3$ as $I_1^3 \leq 2\alpha_d\|\rho_1\|_{\infty}r\E\lt[  G^{p-1} \rt]$. Combining the above estimates, we have
\[
I_1 \leq C\lt( \|\rho_1\|_{\infty}+ \|\rho_2\|_{\infty}\rt)r\E\lt[ G^{p-1} \rt] \leq C\lt( \|\rho_1\|_{\infty}+ \|\rho_2\|_{\infty}\rt)r\E\lt[ G^p \rt]^{(p-1)/p},
\]
where $C > 0$ only depends on $d$. 
 
\noindent $\diamond$ Estimate of $I_2$: 
We decompose $I_2$ as
$$\begin{aligned}
I_2&= C\E\lt[|X-Y|\lt(\frac{1}{|X-\overline{X}|^d}+\frac{1}{|Y-\overline{Y}|^d}\rt)\mb_{|X-\overline{X}|\wedge|Y-\overline{Y}|> r}G^{p-1} \rt]\\
&\quad +C\E\lt[|\overline{X}-\overline{Y}|\lt(\frac{1}{|X-\overline{X}|^d}+\frac{1}{|Y-\overline{Y}|^d}\rt)\mb_{|X-\overline{X}|\wedge|Y-\overline{Y}|> r}G^{p-1} \rt]\\
&=:I_2^1+I_2^2.
\end{aligned}$$
First we easily obtain 
$$\begin{aligned}
I_2^1&=C\E_{(X,V),(Y,W)}\lt[ G^p\E_{\overline{X},\overline{Y}}\lt[  \lt(\frac{1}{|X-\overline{X}|^d}+\frac{1}{|Y-\overline{Y}|^d}\rt)\mb_{|X-\overline{X}|\wedge|Y-\overline{Y}|> r} \rt]    \rt]\\
& \leq C\E_{(X,V),(Y,W)}\lt[ G^p \lt(\int\frac{1}{|X-x|^d}\mathbf{1}_{|X-x|\geq r}\,\rho_1(dx)+\int\frac{1}{|Y-y|^d}\mathbf{1}_{|Y-y|\geq r}\,\rho_2(dy)\rt)  \rt].
\end{aligned}$$
We then consider two cases: $r > 1$ and $0 < r \leq 1$. For $r \leq 1$, we get
$$\begin{aligned}
\int\frac{1}{|X-x|^d}\mathbf{1}_{|X-x|\geq r}\,\rho_1(dx)&=\int\frac{1}{|X-x|^d}\mathbf{1}_{|X-x|> 1}\,\rho_1(dx)+\int\frac{1}{|X-x|^d}\mathbf{1}_{|X-x| \in [r,1]}\,\rho_1(dx)\\
&\leq \|\rho_1\|_1 + \alpha_d\|\rho_1\|_{\infty}\int_r^1u^{-1}du \\
&\leq 1- \alpha_d\|\rho_1\|_{\infty}\ln^- r.
\end{aligned}$$
For the case $r > 1$, it is clear to obtain
\[
\int\frac{1}{|X-x|^d}\mathbf{1}_{|X-x|\geq r}\,\rho_1(dx) \leq \|\rho_1\|_1 = 1.
\]
This yields
\[
I_2^1\leq C\lt(\|\rho_1\|_{\infty}+\|\rho_2\|_{\infty}\rt)\E\lt[G^p\rt]\lt(1 - \ln^- r \rt),
\]
where $C > 0$ only depends on $d$. For the term $I_2^2$, we use Holder inequality to find
$$\begin{aligned}
I_2^2&=C\E\lt[|\overline{X}-\overline{Y}|\lt(\frac{1}{|X-\overline{X}|^d}+\frac{1}{|Y-\overline{Y}|^d}\rt)^{1/p}\lt(\lt(\frac{1}{|X-\overline{X}|^d}+\frac{1}{|Y-\overline{Y}|^d}\rt)^{1/p}G\rt)^{p-1} \hspace{-0.35cm}\mb_{|X-\overline{X}|\wedge|Y-\overline{Y}|> r}\rt]\\
& \leq C\E\lt[|\overline{X}-\overline{Y}|^p\lt(\frac{1}{|X-\overline{X}|^d}+\frac{1}{|Y-\overline{Y}|^d}\rt)\mb_{|X-\overline{X}|\wedge|Y-\overline{Y}|> r}\rt]^{1/p}\\
&\qquad  \qquad \times \E\lt[\lt(\frac{1}{|X-\overline{X}|^d}+\frac{1}{|Y-\overline{Y}|^d}\rt) \mb_{|X-\overline{X}|\wedge|Y-\overline{Y}|> r} G^{p} \rt]^{(p-1)/p}.
\end{aligned}$$
Similarly as before, we take the expectations on $(X,Y)$ and $(\overline{X}, \overline{Y})$ for the first and second expectations in the above, respectively, to find
$$\begin{aligned}
I_2^2&\leq C\lt(\|\rho_1\|_{\infty}+\|\rho_2\|_{\infty}\rt)\lt(1 - \ln^- r \rt) \E\lt[ |X-Y|^p\rt]^{1/p}\E\lt[G^p\rt]^{(p-1)/p}\cr
& \leq C\lt(\|\rho_1\|_{\infty}+\|\rho_2\|_{\infty}\rt)\lt(1 - \ln^- r \rt) \E\lt[ G^p\rt],
\end{aligned}$$
due to $|X- Y |\leq G$, where $C > 0$ only depends on $d$. Thus, by putting all those estimates together and using Holder's inequality, we have for any $r>0$
$$
\begin{aligned}
&\E\lt[|F(X-\overline{X})-F(Y-\overline{Y})|G^{p-1}  \rt]\cr
&\quad \leq C\lt(\|\rho_1\|_{\infty}+\|\rho_2\|_{\infty}\rt)\lt(1 - \ln^- r \rt) \E\lt[ G^p\rt] +  C\lt( \|\rho_1\|_{\infty}+ \|\rho_2\|_{\infty}\rt)r\E\lt[ G^{p} \rt]^{(p-1)/p}.
\end{aligned}
$$
Finally, we choose $r=\E\lt[G^p\rt]^{1/p}$ to obtain the desired result.

{\it Estimate for the cut-off force field.-} Note that since $F^N$ is continuous we have
$$\begin{aligned}
&\E\lt[|F^N(X-\overline{X})-F^N(Y-\overline{Y})|G_N^{p-1}  \rt]\\
&\leq C\E\lt[\lt(|X-Y|+|\overline{X}-\overline{Y}|\rt)\lt(|\nabla F^N(X-\overline{X})|+|\nabla F^N(Y-\overline{Y})|\rt)G_N^{p-1} \rt]\\
&\leq C\E\lt[|X-Y|\lt(|\nabla F^N(X-\overline{X})|+|\nabla F^N(Y-\overline{Y})|\rt)G_N^{p-1} \rt]\\
&\quad +C\E\lt[|\overline{X}-\overline{Y}|\lt(|\nabla F^N(X-\overline{X})|+|\nabla F^N(Y-\overline{Y})|\rt)G_N^{p-1} \rt]\\
&=:J_1+J_2.
\end{aligned}$$
As in the proof above, we first easily get
$$\begin{aligned}
J_1&\leq C\E_{(X,V),(Y,W)}\lt[|X-Y|G_N^{p-1}\E_{\overline{X},\overline{Y}}\lt[|\nabla F^N(X-\overline{X})|+|\nabla F^N(Y-\overline{Y})|\rt] \rt]\\
& \leq C\lt( \|\rho_1\|_{\infty}+ \|\rho_2\|_{\infty}\rt) \ln N \,\E\lt[|X-Y|G_N^{p-1}\rt]\cr
&\leq C\lt( \|\rho_1\|_{\infty}+ \|\rho_2\|_{\infty}\rt) \sqrt{\ln N} \,\E\lt[G_N^p\rt],
\end{aligned}$$
where we used Lemma \ref{lem_useful} and $\sqrt{\ln N}|X-Y| \leq G_N$.  For the term $J_2$, we again use the similar argument as before to find
$$\begin{aligned}
J_2& \leq C\E\lt[|\overline{X}-\overline{Y}|^p\lt(|\nabla F^N(X-\overline{X})|+|\nabla F^N(Y-\overline{Y})|\rt)\rt]^{1/p}\\
& \qquad \quad \times \E\lt[\lt(|\nabla F^N(X-\overline{X})|+|\nabla F^N(Y-\overline{Y})|\rt)G_N^{p} \rt]^{(p-1)/p}.
\end{aligned}$$
Taking the expectations on $(X,Y)$ in the first expectation and on $(\overline{X},\overline{Y})$ in the second one leads to the desired result. 
\end{proof}

We next estimate the error between solutions to the nonlinear SDE and the one with cut-off given by 
\begin{itemize}
\item {\it Nonlinear SDE:}
\bq\label{eq:VPNLSDE1}
\left\{ \begin{array}{ll}
dY_t=W_t\,dt, & \\[2mm]
\displaystyle dW_t= (F * \rho_t)(Y_t)\,dt + \sqrt{2\sigma} dB_t, \quad  \rho_t=\mathcal{L}(Y_t), &\\[2mm]
\LL (Y_0, W_0)= f_0, &  
\end{array} \right.
\eq
\item {\it Nonlinear SDE with cut-off:}
\begin{equation}\label{eq:NLSDE_CO1}
\left\{ \begin{array}{ll}
dY^N_t=W^N_t\,dt, & \\[2mm]
\displaystyle dW^N_t= (F^N * \rho_t^N)(Y_t^N)\,dt + \sqrt{2\sigma} dB_t, \quad  \rho_t^N=\mathcal{L}(Y^N_t), &\\[2mm]
\LL (Y_0^N, W_0^N)= f_0, &  
\end{array} \right.
\end{equation}
\end{itemize}
Here $\rho^N_t = \int_{\R^d} f^N_t\,dv$ and $f_t^N$ is the global-in-time weak solution of the equation \eqref{reg_VPFP}. As mentioned in Introduction, for fixed $N > 0$, the global existence and uniqueness of solutions to \eqref{eq:NLSDE_CO1} is ensured due to classical SDE theory. At the moment, we assume the existence of solutions to the SDE \eqref{eq:VPNLSDE1} and its associated PDE \eqref{eq:VPFP} up to a given time $T>0$. We will give the details of that in Section \ref{sec:well}. 

\begin{proposition}\label{prop:Cau} For a given $T > 0$, let $(Y_t, W_t)$ and $(Y_t^N, W_t^N)$ be the solutions to the equations  \eqref{eq:VPNLSDE1} and \eqref{eq:NLSDE_CO1} for the same initial condition on the time interval $[0,T]$, respectively. Suppose that $\rho_t, \rho_t^N \in L^1(0,T;L^\infty(\R^d))$ and $\|\rho_t^N\|_{L^1(0,T;L^\infty)} \leq C$ with $C > 0$ independent of $N$. Then, for $N \geq e$ and $p\geq1$, we have 
\[
\E\lt[\sup_{0\leq t \leq T}\lt(\sqrt{\ln N}|Y_t^N-Y_t|+|W_t^N-W_t|\rt)^p \rt]^{1/p} \leq CN^{-\delta}\exp\lt(C\sqrt{\ln N}\rt),
\]
where $C$ is a positive constant independent of $p$ and $N$.
\end{proposition}

\begin{proof} For $p\geq1$, we set 
\[
D_s:= \sqrt{\ln N}|Y_s^N-Y_s|+|W_s^N-W_s| \quad \mbox{and} \quad \phi_t^N := \sup_{0 \leq s \leq t} D_s^p.
\]
Then we estimate $\phi_t^N$ as
$$ \begin{aligned}
\phi_t^N & \leq p\int_0^t \lt(\sqrt{\ln N}|W_s^N-W_s|+|F^N*\rho_s^N(Y_s^N)-F*\rho_s(Y_s)|\rt)D_s^{p-1}ds\\
& \leq p\int_0^t \sqrt{\ln N}D_s^{p}+ \lt|F^N* \lt(\rho_s^N(Y_s^N)- \rho_s(Y_s)\rt)\rt|D_s^{p-1} + |(F^N - F)*\rho_s(Y_s)|D_s^{p-1}\,ds.\cr
\end{aligned}$$
Taking the expectation on both sides of the above inequality and using Fubini's Theorem, we obtain
$$\begin{aligned}
\mathbb{E}\lt[\phi_t^N\rt] & \leq p\sqrt{\ln N}\int_0^t \mathbb{E}\lt[\phi_s^N\rt]ds+p\int_0^t \mathbb{E}\lt[|F^N* (\rho_s^N(Y_s^N)- \rho_s(Y_s))|D_s^{p-1}\rt]ds\\
&\quad + p \int_0^t \mathbb{E}\lt[|(F^N - F)*\rho_s(Y_s)|D_s^{p-1}\rt]ds\\
&=: I_1+I_2+I_3,
\end{aligned}$$
where we can directly use the cut-off force field estimate in Lemma \ref{lem:log_lip} to estimate $I_2$ as
\[
I_2  \leq  C p  \sqrt{\ln N} \int_0^t \lt( \|\rho^N_s\|_{\infty}+ \|\rho_s\|_{\infty}\rt) \E\lt[D_s^p\rt]ds,
\]
where $C > 0$ only depends on $d$. For the estimate of $I_3$, we easily find
\begin{align*}
\begin{aligned}
|(F^N - F)*\rho_s(Y_s)|&\leq\int_{|Y_s-y|\leq N^{-\delta}}|Y_s-y|\bigl(|Y_s-y|^{-d}-N^{d\delta}\bigr)\rho_s(dy)\cr
&\leq\int_{|Y_s-y|\leq N^{-\delta}}|Y_s-y|^{-(d-1)}\rho_s(dy) \leq \alpha_d\| \rho_s \|_{\infty}N^{-\delta}.
\end{aligned}
\end{align*}
This yields
$$\begin{aligned}
I_3 &\leq p \alpha_d N^{-\delta}\int_0^t\|\rho_s \|_{\infty}\E\lt[ D_s^{p-1}\rt] \,ds
\leq  p \alpha_d N^{-\delta}\int_0^t\|\rho_s \|_{\infty}\E\lt[ \phi_s^N\rt]^{(p-1)/p} \,ds\cr
&\leq  (p-1)\alpha_d\int_0^t\|\rho_s \|_{\infty}\E\lt[ \phi_s^N\rt] \,ds +  \alpha_d N^{-p\delta}\int_0^t\|\rho_s\|_\infty \,ds,
\end{aligned}$$
where we used Young's inequality for the last inequality. We then combine the above estimates to have
\[
 \E\lt[\phi_t^N\rt] \leq pC \sqrt{\ln N}\int_0^t \lt( 1 + \|\rho_s\|_{\infty} + \|\rho^N_s\|_{\infty}\rt)\E\lt[\phi_s^N \rt]ds + \alpha_d N^{-p\delta}\int_0^t\|\rho_s \|_{\infty} \,ds,
\]
where $C >0$ only depends on $d$. We now apply Lemma \ref{lem_gron}.1 with 
\[
f(t) = \mathbb{E}\lt[\phi_t^N\rt], \quad g(t) = pC\sqrt{\ln N}\lt( 1 + \|\rho_s\|_{\infty} + \|\rho^N_s\|_{\infty}\rt), \quad \mbox{and} \quad h(t) = \alpha_d N^{-p\delta}\int_0^t\|\rho_s \|_{\infty} \,ds,
\]
to obtain
\[
\begin{aligned}
\mathbb{E}\lt[\phi_t^N\rt]& \leq \alpha_d N^{-p\delta}\int_0^t \|\rho_s\|_\infty \exp\lt(p C\sqrt{\ln N}\int_s^t \lt(1 + \|\rho_u\|_\infty + \|\rho^{N}_u\|_\infty \rt) du\rt)ds\\
&\leq Ct N^{-p\delta}\exp\lt(p C\sqrt{\ln N}	t\rt),
\end{aligned}
\]
due to the uniform bound assumption on both $\rho_t$ and $\rho_t^N$ in $L^1(0,T;L^\infty(\R^d))$. This completes the proof. 
\end{proof}

\begin{remark} Note that for any $p >0$
\[
\exp\lt(\sqrt{\ln N}\rt) = \exp\lt(\frac{\ln N}{\sqrt{\ln N}}\rt) = N^{\frac{1}{\sqrt{\ln N}}} = o(N^{p}).
\]
\end{remark}

Let us define the functional $J^N: \R^{4dN} \to \R_+$ by
$$ 
J^N:(\mathcal{X},\mathcal{V},\mathcal{Y},\mathcal{W})\in\R^{4dN} \mapsto 1\wedge\bigl(\sqrt{\ln N}N^{\delta}|\mathcal{X}-\mathcal{Y}|_{\infty}+N^{\delta}|\mathcal{V}-\mathcal{W}|_{\infty}\bigr). 
$$	
Then by using the same argument as in \cite[Theorem 4.2]{LP}, which is based on Gronwall lemma, Lemma \ref{lem:LLN} and Markov's inequality, we obtain the following estimate on $J^N$.
\begin{lemma}\label{lem_j} Let $(\XX_t^N,\VV_t^N)_{t\geq 0}$ be a solution to \eqref{eq:Nps_CO} and let $(\YY_t^N,\WW_t^N)_{t\geq 0}$ be solutions to \eqref{eq:NLps_CO} with the same independent identically distributed initial conditions. 
Assume that 
\[
\sup_{N\in \N} \int_0^T \|\rho_t^N\|_{\infty}\,dt < \infty.
\]
Then, for any $\delta \in (0,\frac{1}{d})$ and $\beta > 0$, there exists $C_\beta>0$ such that
$$
\mathbb{P}\lt(\sup_{0 \leq t \leq T} J_t^N \geq 1 \rt) \leq \frac{C_\beta}{N^\beta},
$$
where $J_t^N:=J^N(\XX_t^N,\VV_t^N,\YY_t^N,\WW_t^N)$. 
\end{lemma}


\begin{proof}[Proof of Theorem \ref{thm:PropChao}] For any $p \geq 1$, we first estimate $\WW_p(\mu_t^N, f_t)$ as
\[
\WW_p(\mu_t^N, f_t)\leq \WW_{\infty}(\mu_t^N,\nu_t^N)+\WW_p(\nu_t^N,f_t^N)+\WW_p(f_t^N,f_t),
\]
where the empirical measure $\nu_t^N$ is the associated to $N$ copies solutions to the nonlinear SDE with cut-off \eqref{eq:NLSDE_CO}. It follows from Proposition \ref{prop:Cau} that for any $\gamma<\delta$ and $t\in [0,T]$, we find 
\[
\WW_p(f_t^N,f_t)\leq C N^{-\delta} e^{C\sqrt{\ln N}}\leq N^{-\gamma} \quad \mbox{for $N$ large enough}. 
\]
This implies
\bq\label{est_l1}
\mathbb{P} \lt( \WW_p(\mu_t^N, f_t)\geq 3 N^{-\gamma}\rt)\leq \mathbb{P} \lt( \WW_{\infty}(\mu_t^N, \nu_t^N)\geq N^{-\gamma}\rt)+\mathbb{P} \lt( \WW_p(\nu_t^N, f^N_t)\geq N^{-\gamma}\rt).
\eq
Note that under the event $\{J_t^N< 1 \}$ we get
\[
\WW_{\infty}(\mu_t^N,\nu_t^N)< N^{-\delta},
\]
thus by using Lemma \ref{lem_j}, we obtain
$$\begin{aligned}
 \mathbb{P} \bigl( \WW_{\infty}(\mu_t^N, \nu_t^N)\geq N^{-\gamma}\bigr) &\leq \mathbb{P} \bigl( \WW_{\infty}(\mu_t^N, \nu_t^N)\geq N^{-\delta}\bigr)\leq \mathbb{P} \bigl( J_t^N \geq 1 \bigr)\leq \frac{C_\beta}{N^\beta},
\end{aligned}$$
where $C_\beta > 0$ is independent of $N$. For the estimate of last term in \eqref{est_l1}, we use Proposition \ref{prop_fg} with $x=N^{-p\gamma}$ to have
\[
\mathbb{P} \lt( \WW_p^p(\nu_t^N, f^N_t)\geq N^{-p\gamma}\rt) \leq CN^{1 -\frac{(1 - p\gamma)(q-\e)}{p}}+a(N,N^{-p\gamma}).
\]
Combining the above estimates concludes the desired result. 
\end{proof}

%
%
%
%
\section{Well-posedness of nonlinear SDE}\label{sec:well}

In this section, we study the well-posedness of nonlinear SDE \eqref{eq:VPNLSDE1} which is associated to the VPFP equation \eqref{eq:VPFP}. For this, we use the nonlinear SDEs with cut-off given by
\begin{equation}\label{eq:NLSDE_CO}
\left\{ \begin{array}{ll}
dY^N_t=W^N_t\,dt, & \\[2mm]
\displaystyle dW^N_t= (F^N * \rho_t^N)(Y_t^N)\,dt + \sqrt{2\sigma} dB_t, \quad  \rho_t^N=\mathcal{L}(Y^N_t), &\\[2mm]
\end{array} \right.
\end{equation}
where $\rho^N_t$ is the spatial density of solution to \eqref{reg_VPFP}. We first show the uniform-in-$N$ estimate of spatial density $\rho_t^N$ in $L^1(0,T;L^\infty(\R^d))$. 

\begin{lemma}\label{lem:Linfbound} Let $T > 0$. Assume that the initial data $f_0$ satisfies
$f_0 \in (L^1 \cap L^\infty)(\R^{2d})$.
Then there exists a unique weak solution $f_t^N$ to the system \eqref{reg_VPFP} with the initial data $f_0$, such that $f^N \in L^\infty(0,T;(L^1 \cap L^\infty)(\R^{2d}))$. Furthermore, if we assume that for some $C > 0$
\[
f_0(x,v)\leq C \left\langle v \right\rangle^{-\gamma} \quad \mbox{where} \quad \lal v \ral = \sqrt{1 + |v|^2},
\]
then, for $\gamma >d$, there exists a time $T \geq T_* > 0$ such that
\[
\sup_{N\geq 1} \sup_{t\in [0,T_*]}\|\rho_t^N\|_{\infty}<\infty,
\]
where $\rho_t^N$ denotes the spatial density of the law of solution at time $t$ to equation \eqref{eq:NLSDE_CO}.
\end{lemma}
\begin{proof} Since the existence and uniqueness of solutions $f^N_t$ to the equation \eqref{reg_VPFP} is classical due to the regularity of the force fields, we only focus on the uniform-in-$N$ estimate of the spatial density $\rho_t^N$ in the rest of the proof. We divide the proof into two steps.

$\bullet$ {\bf Stetp A (Feynman-Kac's representation formula)} 
Let $(\chi_\e)_{\e>0}$ be a familiy of mollifying kernels. First, we notice that $f^N_t=\LL(Y_t^N,W_t^N)$, the law of solution to \eqref{eq:NLSDE_CO}, is a solution in the sense of distributions to 
\[
\partial_t f^N_t+v\cdot \nabla_x f^N_t+ (F^N*\rho_t^N) \cdot \nabla_v f^N_t=\sigma\Delta_v f^N_t, \quad (x,v) \in \R^d \times \R^d, t > 0,
\]
with the initial data $f_0=\LL(Y_0^N,W_0^N)$. Denote by $f_t^{N,\e}$ the solution to the same equation with initial condition $f_0*\chi_\e$. Since the coefficients of the above equation are Lipschitz and locally bounded, classical existence theory guarantees the global existence and uniqueness of strong solutions. We now fix $t \ge 0$ and consider the following ``backward'' stochastic integral equations:
\begin{equation*} 
Y^{x,v}_s  = x - \int_0^s W^{x,v}_u \,du, \qquad
W^{x,v}_s  = v - \int_0^s F^N*\rho^{N,\e}_{t-u}(Y^{x,v}_{u})\, du  + \sqrt{2\sigma} \,B_s ,
\end{equation*}
It is classical that there exists a unique strong solutions to the above equations due to the strong regularity of the force fields. We next set  
\[
\theta_s :=  f^{N,\e}_{t-s} \bigl(Y_s^{x,v},W_s^{x,v} \bigr),
\]
and apply Ito's rule to $\theta$ to find
$$\begin{aligned}
\theta_s & = \theta_0+\int_0^s-\lt( \partial_u f^{N,\e}_{t-u}(Y_u^{x,v},W_u^{x,v})-\left\langle \nabla_x f^{N,\e}_{t-u}(Y_u^{x,v},W_u^{x,v}), W_u^{x,v}\right\rangle\rt)du \cr
&\quad -\int_0^s\left\langle \nabla_v  f^{N,\e}_{t-u}(Y_u^{x,v},W_u^{x,v}), F^N*\rho^{N,\e}_{t-u}(Y^{x,v}_{u})\right\rangle du +\sqrt{2\sigma} \int_0^s\left\langle \nabla_v  f^{N,\e}_{t-u}(Y_u^{x,v},W_u^{x,v}),  dB_u\right\rangle \\
&\quad + \sigma\int_0^s \Delta_v f^{N,\e}_{t-u}(Y_u^{x,v},W_u^{x,v})\,du.
\end{aligned}$$
Taking the expectation to the above equation together with 
\[
\E\lt[\int_0^s\left\langle \nabla_v  f^{N,\e}_{t-u}(Y_u^{x,v},W_u^{x,v}), dB_u\right\rangle\rt]=0
\]
yields
$$\begin{aligned}
&\E\lt[\theta_s\rt]=\E\lt[\theta_0\rt]\\
&+\int_0^s \E\lt[\lt(\underbrace{-\pa_u f^{N,\e}_{t-u}-v\cdot\nabla_x f^{N,\e}_{t-u}-(F^N*\rho_{t-u}^{N,\e})(x)\cdot\nabla_v f_{t-u}^{N,\e} +\sigma\Delta_v f^{N,\e}_{t-u}(x,v)}_{=0}\rt)(Y_u^{x,v},W_u^{x,v})\rt]du.
\end{aligned}$$
Thus, finally, we choose $s=t$ to have 
\[
f^{N,\e}_t(x,v)=\E\lt[f_0*\chi_\e(Y_t^{x,v},W_t^{x,v})\rt].
\]
$\bullet$ {\bf Step B (Uniform-in-$N$ estimate)} It follows from the previous bound that 
\[
f^{N,\e}_t(x,v)\leq C\E\lt[\left\langle v- \int_0^t F^N*\rho^{N,\e}_{t-u}(Y^{x,v}_{u})\, du  + \sqrt{2\sigma} \,B_t  \right\rangle^{-\gamma} \rt].
\]
Note that for all $v,w\in \R^d$ and $\gamma \geq 1$
\[
\left\langle v-w \right\rangle^{-1}\leq \sqrt{2}\left\langle v\right\rangle^{-1}\left\langle w\right\rangle \quad \mbox{and} \quad \left\langle v+ w\right\rangle^{\gamma}\leq C_{\gamma}\lt(1+|v|^{\gamma}+|w|^{\gamma}\rt).
\]
Using those facts, we get
\[
f^{N,\e}_t(x,v)\leq C \left\langle v \right\rangle^{-\gamma}  \E\lt[   \left\langle - \int_0^t F^N*\rho^{N,\e}_{t-u}(Y^{x,v}_{u})\, du  + \sqrt{2\sigma} \,B_t  \right\rangle^{\gamma} \rt],
\]
and further we find for $\gamma > d$ 
$$\begin{aligned}
\|\rho_t^{N,\e}\|_{\infty} &\leq C \lt( 1 +(\sqrt{2\sigma})^{\gamma}\E\lt[|B_t|^{\gamma}\rt]+ \E\lt[\lt|\int_0^t F^N*\rho^{N,\e}_{t-u}(Y^{x,v}_{u})\, du\rt|^{\gamma}\rt]\rt)\cr
&\leq C\lt(1 + t^{\gamma} + \E\lt[\lt|\int_0^t F^N*\rho^{N,\e}_{t-u}(Y^{x,v}_{u})\, du\rt|^{\gamma}\rt] \rt).
\end{aligned}$$
For the estimate of the last term in the above inequality, we use 
$$\begin{aligned}
\sup_{x\in \R^d}|F^N*\rho^{N,\e}(x)| & \leq \sup_{x\in \R^d}\int_{|x-y|\leq 1}|F^N(x-y)|\rho^{N,\e}(dy)+\sup_{x\in \R^d}\int_{|x-y|> 1}|F^N(x-y)|\rho^{N,\e}(dy)\\
& \quad \leq \|\rho^{N,\e}\|_{\infty}\int_{|y|\leq 1}|y|^{-(d-1)}dy+\|\rho^{N,\e}\|_1\leq C\|\rho^{N,\e}\|_{1,\infty},
\end{aligned}$$
to find for any fixed time $T > 0$
\[
\|\rho_t^{N,\e}\|_{\infty} \leq C\lt(1 + \int_0^t \|\rho^{N,\e}_s\|_\infty^\gamma\,ds  \rt) \quad \mbox{for} \quad 0 \leq t \leq T,
\]
where $C > 0$ is independent of $N$. Finally, we use Lemma \ref{lem_gron}.3 to have
\[
\sup_{0 \leq t \leq T_*}\|\rho^{N,\e}_t\|_\infty \leq C \quad \mbox{with} \quad T_* < \min\lt\{ \frac{\gamma - 1}{C^\gamma}, T\rt\},
\]
where $C > 0$ is independent of $N$. The result follows from the fact that $\rho^{N,\e}_t$ converges at least weakly star to $\rho_t^N$ in $L^\infty(\R^d)$ as $\e$ goes to $0$.
\end{proof}
In the theorem below, we provide the existence and uniqueness of strong solutions for the system \eqref{eq:VPNLSDE1} and weak solutions for the equation \eqref{eq:VPFP} up to time $T_*$. 
\begin{theorem}\label{thm:ex-uniq} 
	Let $p \geq 1$ and $(Y_0,W_0)$ be independent of $(B_t)_{t\geq 0}$ with law $f_0$. Suppose that $f_0$ satisfies 
	\[
	f_0 \in (L^1 \cap \pp_p)(\R^{2d}) \quad \mbox{and} \quad f_0(x,v)\leq C \left\langle v \right\rangle^{-\gamma},
	\]
	for some $C>0$ and any $p \geq 1$. Then there exists at most one solution to the nonlinear SDE \eqref{eq:VPNLSDE1} where $\rho_t = \int_{\R^d} f_t\,dv$ and $f_t \in L^\infty(0,T_*;(L^1 \cap L^\infty)(\R^{2d})) \cap \mc([0,T_*];\pp_p(\R^{2d}))$ is a unique weak solution to the equation \eqref{eq:VPFP} satisfying
	$\rho_t\in L^{\infty}(0,T_*;L^{\infty}(\R^d))$.
\end{theorem}

\begin{proof}[Proof of Theorem \ref{thm:ex-uniq}] We split the proof into three steps.

$\bullet$ {\bf Step A (Cauchy estimates)} Let $(Y_t^N, W_t^N)$ be the strong solution to the system \eqref{eq:NLSDE_CO} on the time interval $[0,T]$. Then by using a similar argument as in Proposition \ref{prop:Cau}, we can show that for $N,N'\geq e$ 
\[
\E\lt[\sup_{0\leq t \leq T}\lt(\sqrt{\ln N}|Y_t^N-Y_t^{N'}|+|W_t^N-W_t^{N'}|\rt)^p \rt]^{1/p} \leq C\lt( N^{-\delta} + (N')^{-\delta} \rt)\exp\lt(C\sqrt{\ln N}\rt),
\]
where $C$ is a positive constant independent of $p, N'$ and $N$.

$\bullet$ {\bf Step B (Existence)} It follows from the previous step that the sequence $(Y^N_t,W^N_t)_{N \in \mathbb{N}}$ of solution to \eqref{eq:NLSDE_CO}  is a Cauchy sequence. Thus there exists a limit process $(Y_t,W_t)_{t \in [0,T]}$ such that $(Y^N_t,W^N_t) \to (Y_t,W_t)$ as $N \to \infty$ in $L^p(\R^{2d} \times (0,T))$. Moreover, denoting by $(f_t^N)_{N \in \mathbb{N}}$ the sequence of the law of solution to \eqref{eq:NLSDE_CO}, we find
$$\begin{aligned}
\sup_{0 \leq t \leq T} \WW_p(f_t^N,f_t^{N'})&\leq \E\lt[\sup_{0\leq t \leq T}\lt(\sqrt{\ln N}|Y_t^N-Y_t^{N'}|+|W_t^N-W_t^{N'}|\rt)^p \rt]^{1/p} \cr
&\leq C\lt( N^{-\delta} + (N')^{-\delta} \rt)\exp\lt(\sqrt{\ln N} \rt),
\end{aligned}$$
where $C$ is a positive constant independent of $p$ and $N$. This deduces that $(f_t^N)_{N \in \mathbb{N}}$ converges weakly to some $f_t \in C([0,T];\mathcal{P}_p(\R^d))$ which is the law of $(Y_t,W_t)_{t \in [0,T]}$. It now remains to prove that this process is indeed a solution to \eqref{eq:VPNLSDE1}. In order to check this, it is sufficient to prove that 
$(F^N*\rho_t^N(Y_t^N))_{t\in [0,T]}$ converges $\mathbb{P}$ almost surely (up to a subsequence) to $(F*\rho_t(Y_t))_{t\in [0,T]}$. It follows from Proposition \ref{lem:log_lip}, Lemma \ref{lem:Linfbound}, and $\|(F^N - F) * \rho_t\|_{\infty} \leq \|F^N - F\|
_1\|\rho_t\|_\infty \leq CN^{-\delta}\|\rho_t\|_\infty$ that 
$$ \begin{aligned}
&\E\lt[\int_0^T \lt|F^N*\rho_t^N(Y_t^N)-F*\rho_t(Y_t)\rt|dt\rt]\cr
&\quad \leq \E\lt[\int_0^T \lt|F^N*\rho_t^N(Y_t^N)-F^N*\rho_t(Y_t)\rt|dt\rt]+\E\lt[\int_0^T\lt|F^N*\rho_t(Y_t)-F*\rho_t(Y_t)\rt|dt\rt]\\
& \quad \leq C\ln N\int_0^T\lt(\|\rho^N_t\|_{\infty}+\|\rho_t\|_{\infty}\rt) \E\lt[|Y_t^N-Y_t|\rt]dt + CN^{-\delta}\int_0^T\|\rho_t\|_{\infty}\,dt\\
& \quad \leq C \frac{\ln N}{N^\delta} e^{C\sqrt {\ln N}} \to 0 \quad \mbox{as} \quad N \to \infty,
\end{aligned}$$
where $C > 0$ is independent of $N$. 
	
$\bullet$ {\bf Step C (Uniqueness)} Let $(Y_t^1,W_t^1)_{t\geq 0}$ and $(Y_t^2,W_t^2)_{t\geq 0}$ be two solutions to \eqref{eq:VPNLSDE1} with the same initial data $(Y_0,W_0)$ such that $\rho^i_t \in L^1(0,T;L^\infty(\R^d)), i=1,2$. Set
\[
\Delta_t := |Y_t^1-Y_t^2|+|W_t^1-W_t^2|.
\]
Then it follows from Proposition \ref{lem:log_lip} that
$$\begin{aligned}
\frac{d}{dt} \E\lt[\Delta_t^p\rt] &\leq p\E\lt[|W_t^1-W_t^2|\Delta_t^{p-1}\rt] + \E\lt[p |F*\rho_t^1(Y_t^1)-F*\rho_t^2(Y_t^2)|\Delta_t^{p-1}\rt]\cr
&\leq p\E\lt[\Delta_t^p\rt] + pC(\|\rho_t^1\|_\infty + \|\rho_t^2\|_\infty) \E\lt[\Delta_t^p\rt]\lt(1 - \frac{1}{p}\ln^-\E\lt[ \Delta_t^p\rt] \rt).
\end{aligned}$$
Set $Q(t) := \E\lt[\Delta_t^p\rt]$, then we get
\[
Q'(t) \leq pQ(t) + Cp(\|\rho_t^1\|_\infty + \|\rho_t^2\|_\infty) Q(t) \lt( 1 - \ln^- Q(t) \rt),
\]
for $Q(t) \leq e$. On the other hand, since $Q_0 = 0$, applying Lemma \ref{lem_gron}.2 yields $Q(t) = 0$ for $t \in [0,T]$. It is very clear that the uniqueness of solutions to \eqref{eq:VPNLSDE1} implies the uniqueness of weak solutions to the equation \eqref{eq:VPFP}.

\end{proof}

%
%
%
%
\section{Vlasov-Fokker-Planck equation with less singular interactions than Newtonian}
The previous strategy can directly be applied for the system \eqref{eq:VPFP} with milder singular interaction forces. To be more precise, let us consider the following nonlinear Vlasov-Fokker-Planck equation with singular interactions:
\bq\label{eq:VPFP2}
\displaystyle \pa_t f_t+v\cdot\nabla_x f_t+ (F_\alpha*\rho_t)\cdot\nabla_v f_t=\sigma\Delta_v f_t, \quad (x,v) \in \R^d \times \R^d, \quad t > 0,
\eq
where $F_\alpha$ satisfies 
\[
|F_\alpha(x)| \leq  \frac{1}{|x|^{\alpha}} \quad \mbox{and} \quad \lt| \nabla F_\alpha(x) \rt| \leq \frac{1}{|x|^{\alpha+1}} \quad \forall \,x \in \R^d \setminus \{0\},
\]
with $F_\alpha(0) = 0$ by definition. Note that $\alpha = d-1$ corresponds to the Newtonian case \eqref{cou}. Concerning the particle approximations, in a similar fashion as before, we consider the following stochastic particle system with cut-off given by 
\begin{equation}\label{eq:Nps_CO_a}
\left\{ \begin{array}{ll}
\displaystyle dX_t^{i,N}=V_t^{i,N}dt,    & \\[2mm]
\displaystyle dV_t^{i,N}=\frac{1}{N}\sum_{j=1}^N F^N_{\delta,\alpha} (X_t^{i,N}-X_t^{j,N})dt+\sqrt{2\sigma} dB^{i,N}_t, & 
\end{array} \right. \quad i=1,\cdots,N, \quad t > 0, 
\end{equation}
where the cut-off interaction potential $F^N_{\delta,\alpha}$ is given by $F^N_{\delta,\alpha}(x) = F_\alpha(x)$ for $|x| \geq N^{-\delta}$ and satisfies 
\[
\lt| F^N_{\delta,\alpha}(x) \rt| \leq N^{\alpha\delta} \quad \mbox{and} \quad \lt| \nabla F^N_{\delta,\alpha}(x) \rt| \leq N^{(\alpha+1)\delta}  \quad \mbox{for } |x| < N^{-\delta}.
\]
Then defining the associated empirical measure
\[
\mu_t^N=\frac{1}{N}\sum_{i=1}\delta_{X_t^{i,N},V_t^{i,N}},
\]
we can state the following result. 

\begin{theorem}\label{thm:PropChao2}
	Let $T > 0$ and $d > 1$. Let $(X_0^{i,N}, V_0^{i,N})_{i=1,\cdots,N}$ be $N$ independent variables with law $f_0$. Let $f_t$ and $f_t^N$ be the solutions to the nonlinear Vlasov-Fokker-Planck equation \eqref{eq:VPFP2} and its corresponding regularization \eqref{reg_VPFP} with $F^N_{\delta,\alpha}$ instead of $F^N_\delta$ respectively, up to time $T > 0$. Assume that $0 \leq \alpha < d/\ell' - 1$ and that $f, f^N \in L^\infty(0,T; (L^1 \cap L^\ell)(\R^{2d})) \cap \mc([0,T]; \pp_q(\R^{2d}))$ with the same initial data $f_0 \in (L^1 \cap L^\ell \cap \pp_q)(\R^{2d})$ for some $q \geq 2$. Furthermore, assume that their respective spatial densities $\rho_t$ and $\rho^N_t$ satisfy
	\[
	\int_0^T \|\rho_t\|_{\ell}\,dt \leq C_0 \quad \mbox{and} \quad \sup_{N\in \mathbb{N}}\int_0^T \|\rho^N_t\|_{\ell}\,dt \leq C_0,
	\]
where $C_0$ only depending on the initial data and $T$. Then, for any $p \in [1,2q)$, $\delta$ satisfying either
	\[
	\frac{\ell'}{d} \leq \delta < \frac{1}{1+\alpha}
	\]
	or
	\[
	\ell' > \frac{d}{2(1+\alpha)} \quad \mbox{and} \quad \delta < \frac{\ell'}{d},
	\]
and $\e \in \lt(0, q - \frac{p}{1 - p \delta} \rt)$, the estimate  	
	\[
	\sup_{0 \leq t \leq T}\mathbb{P}\lt(\WW_p(\mu_t^N,f_t)\geq 3N^{-\delta} \rt) \leq C N^{\frac1p\lt(1 - \frac{(1-p\delta)(q - \e)}{p} \rt)} + C_N \quad \mbox{for $N$ large enough}
	\]
holds for some constant $C>0$ depending only on $d,T,p,q,\e, f_0$, and $C_0$. Here $C_N$ is given by
	\[
	-\log C_N=\begin{cases}
	\displaystyle \frac Cp N^{1 - 2p\delta} & \text{ if } p>d, \\[1mm]
	\displaystyle \frac{CN^{1 - 2p\delta}}{p (\ln(2 + N^{p\delta}))^2} & \text{ if } p=d, \\[4mm] 
	\displaystyle \frac Cp N^{1 - 2d\delta} & \text{ if } 1 \leq p<d.
	\end{cases}
	\]
\end{theorem}
\begin{remark}
	The condition $0 \leq \alpha < d/\ell' - 1$ is required in order to obtain the $L^\ell$ \textit{ a priori} bound on the density $\rho_t$. If $\ell=\infty$ we recover the result for the Newtonian case (Theorem \ref{thm:PropChao}). 
\end{remark}

Set 
\[
l_{\delta,\alpha}^N(x) := 
\left\{ \begin{array}{ll}
\displaystyle \frac{1}{|x|^{\alpha+1}} & \mbox{if } |x| \geq (\alpha+1) N^{-\delta},\\[3mm]
N^{(\alpha+1) \delta} &  \mbox{otherwise},
\end{array} \right.
\]
for $\alpha \in [0,d-1)$. Similarly as before, for notational simplicity, we omit the subscript $\delta$ in $F^N_{\delta,\alpha}$ and $\l^N_{\delta,\alpha}$ in the rest of this section, i.e., $F^N_\alpha = F^N_{\delta,\alpha}$ and $l^N_\alpha = l^N_{\delta,\alpha}$. In the lemma below, we provide the weak-strong gradient estimate and uniform bound estimate of the gradient of force field in the cut-off parameter $N$, which can be obtained in the same manner as Lemma \ref{lem_useful}.

\begin{lemma}\label{lem_useful2} Let $d-1 > \alpha \geq 0$ be given.

1. There exists a constant $C$, which depends only on $d$, such that 
\[
|F^N_\alpha(x) - F^N_\alpha(x+z)| \leq C l^N_\alpha(x) |z|,
\]
for any $x,z \in \R^d$ with $|z|\leq \alpha N^{-\delta}$.

2. There exists a constant $C>0$ independent of $N$ such that
\[
\|l^N_\alpha * \rho\|_\infty \leq C\|\rho\|_{1,\ell} \quad \mbox{and} \quad \|\nabla F^N_\alpha * \rho\|_\infty \leq C\|\rho\|_{1,\ell},
\]
where $\ell > 1$ satisfies $d > (\alpha+1) \ell' $ with $1/\ell' + 1/\ell = 1$.
\end{lemma}
Then slightly modifying the proof of Lemma \ref{lem:LLN}, we have the following lemma.
\begin{lemma}\label{lem:LLN22}
Let $(Y_1,\cdots,Y_N)$ be i.i.d. random variables of the law $\rho\in L^\ell(\R^d)$, define the associated empirical measure $\rho_N=\frac{1}{N}\sum_{i=1}^N\delta_{Y_i}$. Then define $\tilde \e:=2\kappa\delta+(1-d\delta/\ell')\mb_{\ell'> d\delta}$ with $1/\ell + 1/\ell' = 1$. If $\tilde \e<2$ then for all integer $m>\frac{1}{2-\tilde \e}$ there exist $\tilde \gamma_m,C>0$ such that 
\[
\mathbb{E}\lt[\sup_{1 \leq i \leq N}  \left | h*\rho_N(Y_i)-h*\rho(Y_i) \right |^{2m}  \rt]\leq C N^{-\tilde \gamma_m},
\]
where $\tilde \gamma_m=(2-\tilde \e)m-1$ and $h$ is defined as in \eqref{eq_h}.
\end{lemma}
Due to the milder singularity in the interactions than the Newtonian, we can bound the gradient of the force term uniformly in $N$ in Lemma \ref{lem_useful2}. This also enables us not to introduce the different weights in position and velocity for the error estimate between solutions of the corresponding nonlinear SDE and the one with cut-off. More precisely, we have the following proposition which corresponds to Proposition \ref{prop:Cau}. 

\begin{proposition}\label{prop:Cau2}For a given $T > 0$, let $(Y_t, W_t)$ and $(Y_t^N, W_t^N)$ be the solutions to the equations \eqref{eq:NLSDE_CO1} with $F^N_\alpha$ and \eqref{eq:VPNLSDE1} with $F_\alpha$ appeared in \eqref{eq:VPFP2} on the time interval $[0,T]$, respectively. Suppose that $\rho_t, \rho_t^N \in L^1(0,T;L^\ell(\R^d))$ and $\|\rho_t^N\|_{L^1(0,T;L^\ell)} \leq C$ with $C > 0$ independent of $N$. Then, for $p\geq1$ and $d > (\alpha + 1)\ell'$, we have 
\[
\E\lt[\sup_{0\leq t \leq T}\lt(|Y_t^N-Y_t|+|W_t^N-W_t|\rt)^p \rt]^{1/p} \leq CN^{-(d/\ell' - \alpha)\delta},
\]
where $C$ is a positive constant independent of $p$ and $N$.
\end{proposition}
Let us define functional $\tilde J : \R^{4dN} \to \R_+$ by
$$ 
\tilde J:(\mathcal{X},\mathcal{V},\mathcal{Y},\mathcal{W})\in\R^{4dN} \mapsto 1\wedge\bigl(N^{\delta}|\mathcal{X}-\mathcal{Y}|_{\infty}+N^{\delta}|\mathcal{V}-\mathcal{W}|_{\infty}\bigr). 
$$
In the lemma below, we provide the estimate on $\tilde J^N$ whose proof can be obtained by using the similar argument as in Lemma \ref{lem_j}.
\begin{lemma}\label{lem_j2} Let $(\XX_t^N,\VV_t^N)_{t\geq 0}$ be a solution to \eqref{eq:Nps_CO_a} and let $(\YY_t^N,\WW_t^N)_{t\geq 0}$ be N copies solutions to \eqref{eq:NLps_CO} with $F^N_\alpha$ instead of $F^N$ and the same independent identically distributed initial conditions. Define $\tilde J_t^N$ by
$$
\tilde J_t^N:=\tilde J(\XX_t^N,\VV_t^N,\YY_t^N,\WW_t^N).
$$
Suppose $d > (\alpha + 1)\ell'$ and $\|\rho^N\|_{L^1(0,T;L^\ell)} \leq C$. Then for any $\beta>0$ there is a constant $C_\beta>0$ such that it holds
$$
\mathbb{P}\lt(\sup_{0 \leq t \leq T}\tilde J_t^N \geq 1\rt)\leq \frac{C_\beta}{N^{\beta}}.
$$
\end{lemma}
\begin{proof}[Proof of Theorem \ref{thm:PropChao2}] The strategy of the proof is the same as the one of Theorem \ref{thm:PropChao}. We first estimate $\WW_p(\mu_t^N, f_t)$ as
\[
\WW_p(\mu_t^N, f_t)\leq \WW_{\infty}(\mu_t^N,\nu_t^N)+\WW_p(\nu_t^N,f_t^N)+\WW_p(f_t^N,f_t),
\]
for any $p \geq 1$, where the empirical measure $\nu_t^N$ is the associated to $N$ copies solutions to the nonlinear SDE with cut-off with $F^N_\alpha$. We now estimate each of the three terms as follows: \newline 
\begin{itemize}
\item It follows from Proposition \ref{prop:Cau2} that for $t\in [0,T]$, we find that 
\[
\WW_p(f_t^N,f_t)\leq CN^{-(d/\ell' - \alpha)\delta} \leq N^{-\delta} \quad \mbox{for $N$ large enough},
\]
due to $d/\ell' - \alpha > 1$. Then we deduce
\[
\mathbb{P} \lt( \WW_p(\mu_t^N, f_t)\geq 3 N^{-\delta}\rt)\leq \mathbb{P} \lt( \WW_{\infty}(\mu_t^N, \nu_t^N)\geq N^{-\delta}\rt)+\mathbb{P} \lt( \WW_p(\nu_t^N, f^N_t)\geq N^{-\delta}\rt).
\]
\item Note that under the event $\{\tilde J_t^N< 1 \}$ we get
\[
\WW_{\infty}(\mu_t^N,\nu_t^N)< N^{-\delta},
\]
thus by Lemma \ref{lem_j2}, we obtain 
\[
\mathbb{P} \bigl( \WW_{\infty}(\mu_t^N, \nu_t^N)\geq N^{-\delta}\bigr)\leq \mathbb{P} \bigl( \tilde J_t^N \geq 1 \bigr)\leq C_\beta N^{-\beta},
\]
for any $\beta>0$.
\item Finally we use Proposition \ref{prop_fg} with $x=N^{-p\delta}$ to have
\[
\mathbb{P} \lt( \WW_p^p(\nu_t^N, f^N_t)\geq N^{-p\delta}\rt) \leq CN^{1 -\frac{(1 - p\delta)(q-\e)}{p}}+a(N,N^{-p\delta}).
\]
\end{itemize}
This completes the proof.
\end{proof}

%
%
%
%
   
\appendix

\section{Gronwall type inequalities}
In this appendix, we provide several Gronwall type integral and  differential inequalities.
\begin{lemma}\label{lem_gron} 1. Let $f, g, h$ be nonnegative functions satisfying 
\bq\label{eq_gron}
f(t) \leq h(t) + \int_0^t g(s)f(s)\,ds, \quad t \geq 0.
\eq
Then we have
\[
f(t) \leq h(0) e^{\int_0^t g(s)\,ds}+ \int_0^t h'(s)e^{\int_s^t g(\tau)\,d\tau} ds, \quad t \geq 0.
\]

2. Let $f$ be nonnegative function satisfying 
\bq\label{ineq_os}
f'(t) \leq Cf(t)(1-\ln^- f(t)) \quad \mbox{for} \quad t \in [0,T],
\eq
where $C > 0$. If $f_0$ satisfies
\[
f_0 < \min\{\exp\lt(1 - e^{CT}\rt), 1\},
\]
then we have
\[
f(t) \leq \exp\lt(1 - (1 - \ln f_0)e^{-Ct} \rt) \quad \mbox{for} \quad t \in [0,T].
\]
In particular, if $f_0=0$, then $f(t) \equiv 0$ for $t \in [0,T]$.

3. Let $\gamma > 1$ and $f$ be nonnegative function satisfying
\[
f(t) \leq C_0 + C_1 \int_0^t f(s)^\gamma\,ds,
\]
where $C_0,C_1 > 0$. Then we have
\[
f(t) \leq \lt(C_0^{1-\gamma} - \frac{C_1^2}{\gamma-1}t \rt)^{\frac{1}{1-\gamma}} \quad \mbox{for} \quad t < \frac{C_0^{1-\gamma}(\gamma-1)}{C_1}.
\]
\end{lemma}
\begin{proof} 1. Set
\[
F(t) := \int_0^t g(s)f(s)\,ds,
\]
then it follows from \eqref{eq_gron} that
\[
F'(t) = g(t)f(t) \leq g(t)\lt(h(t) + F(t) \rt).
\]
This yields
\[
\lt(F(t) e^{-\int_0^t g(s)\,ds} \rt)' \leq h(t)g(t)e^{-\int_0^t g(s)\,ds} = -h(t)\lt(e^{-\int_0^t g(s)\,ds} \rt)'.
\]
Since $F_0 = 0$, we get
\[
F(t) e^{-\int_0^t g(s)\,ds} \leq -h(t)e^{-\int_0^t g(s)\,ds} + h(0) +  \int_0^t h'(s)e^{-\int_0^s g(\tau)\,d\tau} ds,
\]
i.e.,
\[
F(t)  \leq -h(t) + h(0) e^{\int_0^t g(s)\,ds} + \int_0^t h'(s)e^{\int_s^t g(\tau)\,d\tau} ds. 
\]
Thus we have
\[
f(t) \leq h(t) + F(t) \leq h(0) e^{\int_0^t g(s)\,ds} + \int_0^t h'(s)e^{\int_s^t g(\tau)\,d\tau} ds.
\]

2. We first claim that 
\[
f(t) < 1 \quad \mbox{for} \quad t \in [0,T].
\]
Since $f$ is continuous, there exists a $T_0>0$ such that $f(t) < 1$ for $t \in [0,T_0)$. Set
\[
T^* = \sup\lt\{ t \in [0,T] : f(s) < 1 \quad \mbox{for} \quad s \in [0,t] \rt\}.
\]
Let us assume that $T^* < T$. Then for $t \in [0,T^*)$, we obtain from \eqref{ineq_os} that
\[
(\ln f)' \leq C(1 - \ln f(t)), \quad i.e.,\quad (1 - \ln f)' \geq -C(1 - \ln f(t)).
\]
This yields
\[
1 - \ln f(t) \geq (1 - \ln f_0) e^{-Ct}, \quad i.e., \quad \ln f(t) \leq 1 - (1 - \ln f_0)e^{-Ct}.
\]
Thus we get
\[
f(t) \leq \exp\lt(1 - (1 - \ln f_0)e^{-ct} \rt) \quad \mbox{for} \quad t \in [0,T^*).
\]
We now let $t \to T^*$ in the above inequality to find
\[
1 = \lim_{t \to T^*}f(t) \leq \exp\lt(1 - (1 - \ln f_0)e^{-ct} \rt) < 1,
\]
due to the assumption on the initial data $f_0$. This is a contradiction that $T^* < T$ and implies $T^* = T$. Hence we have $f(t) < 1$ and 
\[
f(t) \leq \exp\lt(1 - (1 - \ln f_0)e^{-ct} \rt) \quad \mbox{for} \quad t \in [0,T].
\]

3. Set
\[
F:= C_0 + C_1\int_0^t f(s)^\gamma\,ds, 
\]
then we get $F(0) = C_0$ and 
\[
F'(t) = C_1f(t)^\gamma \leq C_1 F(t)^\gamma.
\]
This yields
\[
(F^{1-\gamma})' \geq \frac{C_1}{1-\gamma},
\]
and for $t < C_0^{1-\gamma}(\gamma-1)C_1^{-1}$
\[
f(t) \leq F(t) \leq \lt(F(0)^{1-\gamma} - \frac{C_1}{\gamma-1}t \rt)^{\frac{1}{1-\gamma}} = \lt(C_0^{1-\gamma} - \frac{C_1}{\gamma-1}t \rt)^{\frac{1}{1-\gamma}}.
\]
\end{proof}


\section*{Acknowledgements}
JAC was partially supported by the EPSRC grant number EP/P031587/1. YPC was supported by NRF grant(No. 2017R1C1B2012918 and 2017R1A4A1014735) and POSCO Science Fellowship of POSCO TJ Park Foundation. SS was supported by the Fondation des Sciences Math\'ematiques de Paris and Universit\'e Paris-Sciences-et-Lettres. The authors would like to thank Maxime Hauray for many fruitful discussions.

%
%
%
%

\bibliographystyle{abbrv}
\bibliography{VPFP}

\end{document}